\documentclass[12pt, twoside, a4paper]{amsart}
\usepackage{amsmath,amsthm,amsopn,amssymb,a4wide,varioref,amscd}

\usepackage{bm}
\usepackage{hyperref}
\newcommand{\comment}[1]{}

\theoremstyle{theorem}
    \newtheorem{theorem}{Theorem}
    \newtheorem{lemma}[theorem]{Lemma}
    \newtheorem{proposition}[theorem]{Proposition}
    \newtheorem{corollary}[theorem]{Corollary}

\theoremstyle{definition} 
    \newtheorem{definition}[theorem]{Definition}

    \newtheorem{remark}[theorem]{Remark}
    \newtheorem{example}[theorem]{Example}
    \newtheorem{exercise}[theorem]{Exercise}

\def\<{\langle}
\def\>{\rangle}

\def\bar{\overline}

\newcommand\mnote[1]{} 
\newcommand\be{\begin{equation*}}

\newcommand\ee{\end{equation*}}

\newcommand\ben{\begin{equation}}
\newcommand\een{\end{equation}}
\newcommand\bes{\begin{eqnarray*}}
\newcommand\ees{\end{eqnarray*}}

\newcommand\bex{\begin{exercise}}
\newcommand\eex{\end{exercise}}
\newcommand\beg{\begin{example}}
\newcommand\eeg{\end{example}}
\newcommand\benu{\begin{enumerate}}
\newcommand\eenu{\end{enumerate}}
\newcommand\beit{\begin{itemize}}
\newcommand\eeit{\end{itemize}}
\newcommand\berk{\begin{remark}}
\newcommand\eerk{\end{remark}}
\newcommand\bdefn{\begin{defintion}}
\newcommand\edefn{\end{definition}}
\newcommand\bthm{\begin{theorem}}
\newcommand\ethm{\end{theorem}}
\newcommand\bprf{\begin{proof}}
\newcommand\eprf{\end{proof}}
\newcommand\blem{\begin{lemma}}
\newcommand\elem{\end{lemma}}

\newcommand{\sm}{{\raise0.3ex\hbox{$\scriptstyle \setminus$}}}

\def\CHI{\mathchoice%
{\raise2pt\hbox{$\chi$}}%
{\raise2pt\hbox{$\chi$}}%
{\raise1.3pt\hbox{$\scriptstyle\chi$}}%
{\raise0.8pt\hbox{$\scriptscriptstyle\chi$}}}
\def\smalloplus{\raise1pt\hbox{$\,\scriptstyle \oplus\;$}}

\numberwithin{equation}{section}

\begin{document}

\title[A Note on Tetrablock Contractions]{A Note on Tetrablock Contractions}
%
%
\author[Sau]{Haripada Sau}
\address{Department of Mathematics,\\
        Indian Institute of Science,\\
        Bangalore 560012, India}

\email{sau10@math.iisc.ernet.in}
\thanks{The author's research is supported by
University Grants Commission, India via DSA-SAP.}
\date{\today}
\maketitle

\begin{abstract}
A commuting triple of operators $(A,B,P)$ on a Hilbert space $\mathcal{H}$ is called a tetrablock contraction if the closure of the set
$$
E = \{\underline{x}=(x_1,x_2,x_3)\in \mathbb{C}^3:
 1-x_1z-x_2w+x_3zw \neq 0 \text{ whenever }|z| \leq 1\text{ and }|w| \leq 1 \}
$$
is a spectral set. In this paper, we have constructed a functional model and produced a complete unitary invariant for a pure tetrablock contraction. In this construction, the fundamental operators, which are the unique solutions of the operator equations
$$
A-B^*P = D_PX_1D_P \text{ and } B-A^*P=D_PX_2D_P, \text{ where $X_1,X_2 \in \mathcal{B}(\mathcal{D}_P)$},
$$
play a big role.
\\
As a corollary to the functional model, we show that every pure tetrablock isometry $(A,B,P)$ on a Hilbert space $\mathcal{H}$ is unitarily equivalent to $(M_{G_1^*+G_2z}, M_{G_2^*+G_1z},M_z)$ on $H^2_{\mathcal{D}_{P^*}}(\mathbb{D})$, where $G_1$ and $G_2$ are the fundamental operators of $(A^*,B^*,P^*)$.
\\
We prove a Beurling-Lax-Halmos type theorem for a triple of operators $(M_{F_1^*+F_2z},M_{F_2^*+F_1z},M_z)$, where $\mathcal{E}$ is a Hilbert space and $F_1,F_2 \in \mathcal{B}(\mathcal{E})$.
\\
We deal with a natural example of tetrablock contraction on functions space to find out its fundamental operators.
\end{abstract}

\section{introduction}
The set {\bf{tetrablock}} is defined as
$$
E = \{\underline{x}=(x_1,x_2,x_3)\in \mathbb{C}^3:
 1-x_1z-x_2w+x_3zw \neq 0 \text{ whenever }|z| \leq 1\text{ and }|w| \leq 1 \}.
$$
This domain was studied in \cite{awy} and \cite{awy-cor} for its geometric properties. Let $A(E)$ be the algebra of functions holomorphic in $E$ and continuous in $\bar{E}$. The distinguished boundary of $E$ (denoted by $b(E)$), i.e., the shilov boundary with respect to $A(E)$, is found in \cite{awy} and \cite{awy-cor} to be the set
$$
bE = \{\underline{x}=(x_1,x_2,x_3)\in \mathbb{C}^3: x_1=\bar{x_2}x_3, |x_3|=1 \text{ and } |x_2| \leq 1 \}.
$$ The operator theory on tetrablock was first developed in \cite{sir's tetrablock paper}.
\begin{definition}
A triple $(A,B,P)$ of commuting bounded operators on a Hilbert space $\mathcal{H}$ is called a {\em{tetrablock contraction}} if $\overline{E}$ is a spectral set for $(A,B,P)$, i.e. the Taylor joint spectrum of $(A,B,P)$ is contained in $\overline{E}$ and
$$
||f(A,B,P)|| \leq ||f||_{\infty,\overline{E}}=\text{sup}\{ |f(x_1,x_2,x_3)|:(x_1,x_2,x_3) \in \overline{E}\}
$$
for any polynomial $f$ in three variables.
\end{definition}

A {\em{tetrablock unitary}} is a commuting pair of normal operators $(A,B,P)$ such that its Taylor joint spectrum is contained in $bE$.

A {\em{tetrablock isometry}} is the restriction of a tetrablock unitary to a joint invariant subspace. See \cite{sir's tetrablock paper}, for several characterizations of a tetrablock unitary and a tetrablock isometry.

Consider a tetrablock contraction $(A,B,P)$. Then it is easy to see that $P$ is a contraction.
\\
Fundamental equations for a tetrablock contraction are introduced in \cite{sir's tetrablock paper}. And these are
\begin{eqnarray}\label{Maa12}
A-B^*P=D_PF_1D_P, \text{ and }  B-A^*P=D_PF_2D_P
\end{eqnarray}
where $D_P=(I-P^*P)^\frac{1}{2}$ is the defect operator of the contraction $P$ and $\mathcal{D}_P=\overline{Ran}D_P$ and where $F_1,F_2$ are bounded operators on $\mathcal{D}_P$.
Theorem 3.5 in \cite{sir's tetrablock paper} says that the two fundamental equations can be solved and the solutions $F_1$ and $F_2$ are unique. The unique solutions
$F_1$ and $F_2$ of equations (\ref{Maa12}) are called the {\em{fundamental operators}} of the tetrablock contraction $(A,B,P)$. Moreover, $w(F_1)$ and $w(F_2)$ are not greater than $1$, where $w(X)$, for a bounded operator $X$, denotes the numerical radius of $X$.

The adjoint triple $(A^*,B^*,P^*)$ is also a tetrablock contraction as can be seen from the definition. By what we stated above there are unique $G_1,G_2 \in \mathcal{B}(\mathcal{D}_{P^*})$ such that
\begin{eqnarray}\label{Maa13}
A^*-BP^*=D_{P^*}G_1D_{P^*} \text{ and } B^*-AP^*=D_{P^*}G_2D_{P^*}.
\end{eqnarray}
Moreover, $w(G_1)$ and $w(G_2)$ are not greater than $1$.

 In \cite{sir's tetrablock paper}(Theorem 6.1), it is showed that tetrablock is a complete spectral set under the conditions that $F_1$ and $F_2$ satisfy
 \begin{eqnarray}\label{Maa14}
[X_1, X_2] =0 \text{ and } [X_1, X_1^*] = [X_2, X_2^* ]
\end{eqnarray} in place of $X_1$ and $X_2$ respectively. Where $[X_1, X_2]$, for two bounded operators $X_1$ and $X_2$, denotes the commutator of $X_1$ and $X_2$, i.e., the operator $X_1X_2-X_2X_1$. In section 2, we show that if the contraction $P$ has dense range, then commutativity of the fundamental operators $F_1$ and $F_2$ is enough to have a dilation of the tetrablock contraction $(A,B,P)$. In fact, under the same hypothesis we show that $G_1$ and $G_2$ also satisfy (\ref{Maa14}), in place of $X_1$ and $X_2$ respectively. This is the content of Theorem \ref{sthm0}.

 For a Hilbert space $\mathcal{E}$, $H^2_{\mathcal{E}}(\mathbb{D})$ stands for the Hilbert space of $\mathcal{E}$-valued analytic functions on $\mathbb{D}$ with square summable Taylor series co-efficients about the point zero. The space $H^2_{\mathcal{E}}(\mathbb{D})$ is unitarily equivalent to the space $H^2(\mathbb{D}) \otimes \mathcal{E}$ via the map $z^n\xi \to z^n \otimes \xi $, for all $n \geq 0$ and $\xi \in \mathcal{E}$. We shall identify these unitarily equivalent
spaces and use them, without mention, interchangeably as per notational convenience

In \cite{beurling}, Beurling characterized invariant subspaces for the 'multiplication by $z$' operator on the Hardy space $H^2(\mathbb{D})$ of the unit disc. In \cite{lax}, Lax extended Beurling's result to the finite-dimensional vector space valued Hardy space. Then Halmos extended Lax's result to infinite-dimensional vector spaces in \cite{halmos}. The extended result is the following.
\begin{theorem}[Beurling-Lax-Halmos]
Let $0 \neq \mathcal{M}$ be a closed subspace of $H^2_{\mathcal{E}}(\mathbb{D})$. Then $\mathcal{M}$ is invariant under $M_z$ if and only if there exist a Hilbert space $\mathcal{E_*}$ and an inner function $(\mathcal{E_*}, \mathcal{E}, \Theta)$ such that $\mathcal{M} = \Theta H^2_{\mathcal{E_*}}(\mathbb{D})$.
\end{theorem}
In section 3, we prove a Beurling-Lax-Halmos type theorem for a triple of operators, which is the first main result of this paper. More explicitly, given a Hilbert space $\mathcal{E}$ and two bounded operators $F_1,F_2 \in \mathcal{B}(\mathcal{E})$, we shall see that a non-zero closed
subspace $\mathcal{M}$ of $H^2_{\mathcal{E}}(\mathbb{D})$ is invariant under $(M_{F_1^*+F_2z},M_{F_2^*+F_1z},M_z)$ if and only if
\begin{eqnarray*}
&&(F_1^*+F_2z)\Theta(z)=\Theta(z)(G_1+G_2^*z)
\\
&&(F_2^*+F_1z)\Theta(z)=\Theta(z)(G_2+G_1^*z), \text{ for all }z \in \mathbb{D}
\end{eqnarray*}
for some unique $G_1,G_2 \in \mathcal{B}(\mathcal{E}_{*})$, where $(\mathcal{E}_{*},\mathcal{E},\Theta)$ is the Beurling-Lax-Halmos representation of $\mathcal{M}$.
Along the way we shall see that if $F_1$ and $F_2$ are such that $(M_{F_1^*+F_2z},M_{F_2^*+F_1z},M_z)$ on $H^2(\mathcal{E})$ is a tetrablock isometry, then $(M_{G_1+G_2^*z},M_{G_2+G_1^*z},M_z)$ is also a tetrablock isometry on $H^2(\mathcal{E_{*}})$. This is the content of Theorem \ref{sthm1}.

A contraction $P$ on a Hilbert space $\mathcal{H}$ is called {\em{pure}} if ${P^*}^n \to 0$ strongly, i.e., $\Vert{P^*}^nh\rVert^2 \to 0$, for all $h \in \mathcal{H}$. A contraction $P$ is called {\em{completely-non-unitary}} (c.n.u.) if it has no reducing sub-spaces on which its restriction is unitary. A tetrablock contraction $(A,B,P)$ is called {\em{pure tetrablock contraction}} if the contraction $P$ is pure.

Section 4 gives a functional model of pure tetrablock contractions, the second main result of this paper. We shall see that if $(A,B,P)$ is a pure tetrablock contraction on a Hilbert space $\mathcal{H}$, then the operators $A,B$ and $P$ are unitarily equivalent to $P_{\mathcal{H}_P}(I \otimes G_1^* + M_z \otimes G_2)|_{\mathcal{H}_P}, P_{\mathcal{H}_P}(I \otimes G_2^* + M_z \otimes G_1)|_{\mathcal{H}_P}$ and $P_{\mathcal{H}_P}(M_z \otimes I_{\mathcal{D}_{P^*}})|_{\mathcal{H}_P}$ respectively, where $G_1$ and $G_2$ are fundamental operators of $(A^*,B^*,P^*)$ and $\mathcal{H}_P$ is the model space of a pure contraction $P$, as in \cite{Nagy-Foias}. This is the content of Theorem \ref{fm}.

Two equations associated with a contraction $P$ and its defect operators that have been known from the time of Sz.-Nagy and that will come handy are
\begin{eqnarray}\label{Maa8}
PD_P=D_{P^*}P
\end{eqnarray}
and its corresponding adjoint relation
\begin{eqnarray}\label{Maa10}
D_PP^*=P^*D_{P^*}.
\end{eqnarray}
Proof of (\ref{Maa8}) and (\ref{Maa10}) can be found in \cite{Nagy-Foias}(ch. 1, sec. 3).

For a contraction $P$, the {\em{characteristic function}} $\Theta_P$ is defined by
\begin{eqnarray}\label{char}
\Theta_P(z)=[-P+zD_{P^*}(I_\mathcal{H}-zP^*)^{-1}D_P]|_{\mathcal{D}_P}, \text{ for all $z \in \mathbb{D}$}.
\end{eqnarray}
By virtue of (\ref{Maa8}), it follows that, for each $z \in \mathbb{D}$, the operator $\Theta_P(z)$ is an operator from $\mathcal{D}_P$ to $\mathcal{D}_{P^*}$.

In \cite{Nagy-Foias}, Sz.-Nagy and Foias developed the model theory for c.n.u. contractions and found a set of unitary invariants. The set is a singleton set and consists of the characteristic function of the contraction. In section 5, we produce a set of unitary invariants for a pure tetrablock contraction $(A,B,P)$. In this case the set of unitary invariants consists of three members, the characteristic function of $P$ and the two fundamental operators of $(A^*,B^*,P^*)$. This (Theorem \ref{unitary inv}) is the second major result of this paper. The result states that for two pure tetrablock contractions $(A,B,P)$ and $(A',B',P')$ to be unitary equivalent it is necessary and sufficient that the characteristic functions of $P$ and $P'$ coincide and the fundamental operators $(G_1,G_2)$ and $(G'_1,G'_2)$ of $(A,B,P)$ and $(A',B',P')$ respectively, are unitary equivalent by the same unitary that is involved in the coincidence of the characteristic functions of $P$ and $P'$.

It is very hard to compute the fundamental operators of a tetrablock contraction, in general. We now know how important the role of the fundamental operators is to find the functional model of pure tetrablock contractions. So it is important to have a concrete example of fundamental operators and grasp the above model theory by dealing with them. That is exactly what Section 6 does. In other words, we find the fundamental operators $(G_1,G_2)$ of the adjoint of a pure tetrablock isometry $(A,B,P)$ and the unitary operator which unitarizes $(A,B,P)$ to the pure tetrablock isometry $(M_{G_1^*+G_2z}, M_{G_2^*+G_1z},M_z)$ on $H^2_{\mathcal{D}_{P^*}}(\mathbb{D})$.

\section{relations between fundamental operators}
In this section, we prove some important relations between fundamental operators of a tetrablock contraction. But before going to state and proof the main theorem of this section, we have to recall two lemmas, which were proved originally in \cite{sir's tetrablock paper}.
\begin{lemma}\label{tetralem6}
Let $(A,B,P)$ be a tetrablock contraction with commuting fundamental operators $F_1$ and $F_2$. Then
$$
A^*A-B^*B = D_P(F_1^*F_1-F_2^*F_2)D_P.
$$
\end{lemma}
\begin{lemma}\label{tetra}
The fundamental operators $F_1$ and $F_2$ of a tetrablock contraction $(A,B,P)$
are the unique bounded linear operators on $\mathcal{D}_P$ that satisfy the pair
of operator equations
\begin{eqnarray*}
D_PA = X_1D_P + X_2^*D_PP \text{ and } D_P B = X_2D_P + X_1^*D_PP.
\end{eqnarray*}
\end{lemma}
Now we state and prove three relations between the fundamental operators of a tetrablock contraction, which will be used later in this paper.
\begin{lemma}\label{tetralem2}
Let $(A,B,P)$ be a tetrablock contraction on a Hilbert space $\mathcal{H}$ and $F_1,F_2$ and $G_1,G_2$ be fundamental operators of $(A,B,P)$ and $(A^*,B^*,P^*)$ respectively. Then
$$
D_PF_1=(AD_P-D_{P^*}G_2P)|_{\mathcal{D}_P} \text{ and } D_PF_2=(BD_{P}-D_{P^*}G_1P)|_{\mathcal{D_P}}.
$$
\end{lemma}
\begin{proof}
We shall prove only one of the above, proof of the other is similar. For $h \in \mathcal{H}$, we have
\begin{eqnarray*}
(AD_P-D_{P^*}G_2P)D_Ph
&=&
A(I-P^*P)h-(D_{P^*}G_2D_{P^*})Ph
\\
&=&
Ah-AP^*Ph-(B^*-AP^*)Ph
\\
&=&
Ah-AP^*Ph-B^*Ph+AP^*Ph
\\
&=&(A-B^*P)h=(D_P F_1)D_Ph.
\end{eqnarray*}
Hence the proof.
\end{proof}
\begin{lemma}\label{tetralem4}
Let (A,B,P) be a tetrablock contraction on a Hilbert space $\mathcal{H}$ and $F_1, F_2$ and $G_1,G_2$ be fundamental operators of $(A,B,P)$ and $(A^*,B^*,P^*)$ respectively. Then
$$
PF_i=G_i^*P|_{\mathcal{D}_P}, \text{ for $i$=$1$ and $2$}.
$$
\end{lemma}
\begin{proof}
We shall prove only for $i=1$, the proof for $i=2$ is similar. Note that the operators on both sides are from $\mathcal{D}_P$ to $\mathcal{D}_{P^*}$.
Let $h,h' \in \mathcal{H}$ be any element. Then
\begin{eqnarray*}
&&\langle (PF_1-G_1^*P)D_Ph, D_{P^*}h' \rangle
\\
&=&
\langle D_{P^*}PF_1D_Ph,h' \rangle- \langle D_{P^*}G_1^*PD_{P}h,h' \rangle
\\
&=&
\langle P(D_P F_1D_P)h,h' \rangle - \langle (D_{P^*}G_1^*D_{P^*})Ph,h'\rangle
\\
&=&
\langle P(A-B^*P)h,h' \rangle - \langle (A-PB^*)Ph,h' \rangle
\\
&=&
\langle (PA-PB^*P-AP+PB^*P)h,h' \rangle =0.
\end{eqnarray*}
Hence the proof.
\end{proof}
\begin{lemma}\label{tetralem3}
Let $(A,B,P)$ be a tetrablock contraction on a Hilbert space $\mathcal{H}$ and $F_1,F_2$ and $G_1,G_2$ be fundamental operators of $(A,B,P)$ and $(A^*,B^*,P^*)$ respectively. Then
\begin{eqnarray*}
&&(F_1^*D_PD_{P^*}-F_2P^*)|_{\mathcal{D}_{P^*}}=D_PD_{P^*}G_1-P^*G_2^* \text{ and }
\\
&&(F_2^*D_PD_{P^*}-F_1P^*)|_{\mathcal{D}_{P^*}}=D_PD_{P^*}G_2-P^*G_1^*.
\end{eqnarray*}
\end{lemma}
\begin{proof}
For $h \in \mathcal{H}$, we have
\begin{eqnarray*}
&&(F_1^*D_PD_{P^*}-F_2P^*)D_{P^*}h
\\
&=&
F_1^*D_P(I-PP^*)h-F_2P^*D_{P^*}h
\\
&=&
F_1^*D_Ph-F_1^*D_PPP^*h-F_2D_PP^*h
\\
&=&
F_1^*D_Ph-(F_1^*D_PP+F_2D_P)P^*h
\\
&=&
F_1^*D_Ph-D_PBP^*h\;\;\;\;\;\;\;\;\;\;\;\;[\text{ by Lemma }(\ref{tetra})]
\\
&=&
(AD_P-D_{P^*}G_2P)^*h-D_PBP^*h \;\;\;\;[\text{by Lemma \ref{tetralem2}}]
\\
&=&
D_PA^*h-P^*G_2^*D_{P^*}h-D_PBP^*h
\\
&=&
D_P(A^*-BP^*)h-P^*G_2^*D_{P^*}h
\\
&=&
D_PD_{P^*}G_1D_{P^*}h-P^*G_2^*D_{P^*}h
\\
&=&
(D_PD_{P^*}G_1-P^*G_2^*)D_{P^*}h.
\end{eqnarray*}
Proof of the other relation is similar and hence is skipped.
Hence the proof.
\end{proof}
Now we prove the main result of this section.
\begin{theorem}\label{sthm0}
Let $F_1$ and $F_2$ be fundamental operators of a tetrablock contraction $(A,B,P)$ on a Hilbert space $\mathcal{H}$. And let $G_1$ and $G_2$ be fundamental operators of the tetrablock contraction $(A^*,B^*,P^*)$. If $[F_1,F_2]=0$ and $P$ has dense range, then
\begin{enumerate}
\item[(i)] $[F_1,F_1^*]=[F_2,F_2^*]$
\item[(ii)] $[G_1,G_2]=0$ and
\item[(iii)] $[G_1,G_1^*]=[G_2,G_2^*]$.
\end{enumerate}
\begin{proof}
\begin{enumerate}
\item[(i)] From Lemma \ref{tetra} we have $D_PA = F_1D_P + F_2^*D_PP $. This gives after multiplying $F_2$ from left in both sides,
\begin{eqnarray*}
&&F_2D_PA=F_2F_1D_P + F_2F_2^*D_PP
\\
&\Rightarrow& D_PF_2D_PA=D_PF_2F_1D_P + D_PF_2F_2^*D_PP
\\
&\Rightarrow& (B-A^*P)A=D_PF_2F_1D_P + D_PF_2F_2^*D_PP
\\
&\Rightarrow& BA-A^*AP = D_PF_2F_1D_P + D_PF_2F_2^*D_PP.
\end{eqnarray*}
Similarly, multiplying $F_1$ from left in both sides of $ D_P B = F_2D_P + F_1^*D_PP$ and proceeding as above we get
$AB-B^*BP=D_PF_1F_2D_P+D_PF_1F_1^*D_PP$. Subtracting these two equations we get
\begin{eqnarray*}
&& (A^*A-B^*B)P=D_P[F_1,F_2]D_P+D_P(F_1F_1^*-F_2F_2^*)D_PP
\\
&\Rightarrow& D_P(F_1^*F_1-F_2^*F_2)D_PP=D_P[F_1,F_2]D_P+D_P(F_1F_1^*-F_2F_2^*)D_PP \; [\text{ applying Lemma \ref{tetralem6}}.]
\\
&\Rightarrow& D_P([F_1,F_1^*]-[F_2,F_2^*])D_PP=0 \; [\text{ since $[F_1,F_2]=0$.}]
\\
&\Rightarrow& D_P([F_1,F_1^*]-[F_2,F_2^*])D_P=0 \; [\text{ since $RanP$ is dense in $\mathcal{H}$}.]
\\
&\Rightarrow&  [F_1,F_1^*]=[F_2,F_2^*]
\end{eqnarray*}
This completes the proof of part $(i)$ of the theorem.
\item[(ii)] From Lemma \ref{tetralem4}, we have that $PF_i=G_i^*P|_{\mathcal{D}_P}$ for $i=1$ and $2$. So we have
\begin{eqnarray*}
PF_1F_2D_P=G_1^*PF_2D_P&\Rightarrow& PF_2F_1D_P=G_1^*PF_2D_P \; [\text{ since $F_1$ and $F_2$ commute}]
\\
&\Rightarrow& G_2^*G_1^*PD_P=G_1^*G_2^*PD_P \; [\text{ applying Lemma \ref{tetralem4}}]
\\
&\Rightarrow& [G_1^*,G_2^*]D_{P^*}P=0 \Rightarrow [G_1,G_2]=0 \; [\text{ since $RanP$ is dense in $\mathcal{H}$}.]
\end{eqnarray*}
This completes the proof of part $(ii)$ of the theorem.
\item[(iii)] From Lemma \ref{tetralem2}, we have $D_PF_1=(AD_P-D_{P^*}G_2P)|_{\mathcal{D}_P}$. Which gives after multiplying $F_2D_P$ from right in both sides
\begin{eqnarray*}
&& D_PF_1F_2D_P=AD_PF_2D_P-D_{P^*}G_2PF_2D_P
\\
&\Rightarrow& D_PF_1F_2D_P = A(B-A^*P) - D_{P^*}G_2G_2^*PD_P \;[\text{ applying Lemma \ref{tetralem4}}.]
\\
&\Rightarrow& D_PF_1F_2D_P = AB-AA^*P-D_{P^*}G_2G_2^*PD_P.
\end{eqnarray*}
Similarly, multiplying $F_1D_P$ from right in both sides of $D_PF_2=(BD_{P}-D_{P^*}G_1P)|_{\mathcal{D_P}}$, we get $D_PF_2F_1D_P = BA-BB^*P-D_{P^*}G_1G_1^*PD_P$. Subtracting these two equations we get
\begin{eqnarray*}
&&D_P[F_1,F_2]D_P = D_{P^*}(G_1G_1^*-G_2G_2^*)D_{P^*}P - (AA^*-BB^*)P
\\
&\Rightarrow& D_P[F_1,F_2]D_P = D_{P^*}(G_1G_1^*-G_2G_2^*)D_{P^*}P - D_{P^*}(G_1^*G_1-G_2^*G_2)D_{P^*}P \; [\text{ applying Lemma \ref{tetralem6}}]
\\
&\Rightarrow& D_P[F_1,F_2]D_P = D_{P^*}([G_1,G_1^*]-[G_2,G_2^*])D_{P^*}P
\\
&\Rightarrow& D_{P^*}([G_1,G_1^*]-[G_2,G_2^*])D_{P^*}P=0 \; [\text{ since $[F_1,F_2]=0$.}]
\\
&\Rightarrow& [G_1,G_1^*]=[G_2,G_2^*] \; [\text{ since $RanP$ is dense in $\mathcal{H}$.}]
\end{eqnarray*}
This completes the proof of part $(iii)$ of the theorem.
\end{enumerate}
Hence the proof of the theorem.
\end{proof}
\end{theorem}
We would like to mention a corollary to Theorem \ref{sthm0} which gives a sufficient condition of when commutativity of the fundamental operators of $(A,B,P)$ is necessary and sufficient for the commutativity of the fundamental operators of $(A^*,B^*,P^*)$.
\begin{corollary}
Let $(A,B,P)$ be a tetrablock contraction on a Hilbert space $\mathcal{H}$ such that $P$ is invertible. Let $F_1,F_2,G_1$ and $G_2$ be as in Theorem \ref{sthm0}. Then $[F_1,F_2]=0$ if and only if $[G_1,G_2]=0$.
\end{corollary}
\begin{proof}
Suppose that $[F_1,F_2]=0$. Since $P$ has dense range, by part (ii) of Theorem \ref{sthm0}, we get $[G_1,G_2]=0$. Conversely, let $[G_1,G_2]=0$. Since $P$ is invertible, $P^*$ has dense Range too. So applying Theorem \ref{sthm0} for the tetrablock contraction $(A^*,B^*,P^*)$, we get $[F_1,F_2]=0$.
\end{proof}
We conclude this section with another relation between the fundamental operators which will be used in the next section.
\begin{lemma}\label{scor1}
Let $F_1$ and $F_2$ be fundamental operators of a tetrablock contraction $(A,B,P)$ and $G_1$ and $G_2$ be fundamental operators of the tetrablock contraction $(A^*,B^*,P^*)$. Then
\begin{eqnarray}
&&\label{t1}(F_1^*+F_2z)\Theta_{P^*}(z)=\Theta_{P^*}(z)(G_1+G_2^*z)
\\
&&\label{t2}(F_2^*+F_1z)\Theta_{P^*}(z)=\Theta_{P^*}(z)(G_2+G_1^*z), \text{ for all }z \in \mathbb{D}.
\end{eqnarray}
\end{lemma}
\begin{proof}We prove equation (\ref{t1}) only, proof of equation (\ref{t2}) is similar.
\begin{eqnarray*}
&&(F_1^*+F_2z)\Theta_{P^*}(z)
\\
&=& (F_1^*+F_2z) (-P^* + \sum_{n=0}^{\infty}z^{n+1}D_PP^nD_{P^*})
\\
&=& (-F_1^*P^*+\sum_{n=1}^{\infty}z^{n}F_1^*D_PP^{n-1}D_{P^*}) + (-zF_2P^* + \sum_{n=2}^{\infty}z^{n}F_2D_PP^{n-2}D_{P^*} )
\\
&=&
-F_1^*P^*+z(-F_2P^*+F_1^*D_PD_{P^*})+ \sum_{n=2}^{\infty} z^n(F_1^*D_PP^{n-1}D_{P^*}+F_2D_PP^{n-2}D_{P^*})
\\
&=& -F_1^*P^*+z(-F_2P^*+F_1^*D_PD_{P^*})+ \sum_{n=2}^{\infty} z^n (F_1^*D_PP+F_2D_P)P^{n-2}D_{P^*}
\\
&=&
-P^*G_1+z(D_PD_{P^*}G_1-P^*G_2^*) + \sum_{n=2}^{\infty} z^nD_PBP^{n-2}D_{P^*}. \; [\text{ using Lemma \ref{tetra}, \ref{tetralem4} and \ref{tetralem3}.}]
\end{eqnarray*}
On the other hand
\begin{eqnarray*}
&&\Theta_{P^*}(z)(G_1+G_2^*z)
\\
&=&
(-P^* + \sum_{n=0}^{\infty}z^{n+1}D_PP^nD_{P^*})(G_1+G_2^*z)
\\
&=&
(-P^*G_1 + \sum_{n=1}^{\infty}z^{n}D_PP^{n-1}D_{P^*}G_1) + (-zP^*G_2^*+\sum_{n=2}^{\infty}z^{n}D_PP^{n-2}D_{P^*}G_2^*)
\\
&=&
-P^*G_1 + z(D_PD_{P^*}G_1-P^*G_2^*) + \sum_{n=2}^{\infty}z^{n}(D_PP^{n-1}D_{P^*}G_1+D_PP^{n-2}D_{P^*}G_2^*)
\\
&=&
-P^*G_1 + z(D_PD_{P^*}G_1-P^*G_2^*) + \sum_{n=2}^{\infty}z^{n}D_PP^{n-2}(PD_{P^*}G_1+D_PG_2^*)
\\
&=&
-P^*G_1 + z(D_PD_{P^*}G_1-P^*G_2^*) + \sum_{n=2}^{\infty}z^{n}D_PP^{n-2}BD_{P^*}
\\
&=& -P^*G_1+z(D_PD_{P^*}G_1-P^*G_2^*) + \sum_{n=2}^{\infty} z^nD_PBP^{n-2}D_{P^*}.
\end{eqnarray*}
Hence $(F_1^*+F_2z)\Theta_{P^*}(z)=\Theta_{P^*}(z)(G_1+G_2^*z)$ for all $z \in \mathbb{D}$.
Hence the proof.
\end{proof}

\section{beurling-lax-halmos representation for a triple of operators}
In this section, we prove a Beurling-Lax-Halmos type theorem for the triple of operators $(M_{F_1^*+F_2z},M_{F_2^*+F_1z},M_z)$ on $H^2_{\mathcal{E}}(\mathbb{D})$, where $\mathcal{E}$ is a Hilbert space and $F_1,F_2 \in \mathcal{B}(\mathcal{E})$. The triple $(M_{F_1^*+F_2z},M_{F_2^*+F_1z},M_z)$ is not commuting triple in general, but  we have showed that when they commute an interesting thing happens.
\begin{theorem}\label{thm1}\label{sthm1}
Let $F_1,F_2 \in \mathcal{B}(\mathcal{E})$ be two operators. Then $\{0\} \neq \mathcal{M} \subseteq H^2_{\mathcal{E}}(\mathbb{D})$ is $(M_{F_1^*+F_2z},M_{F_2^*+F_1z},M_z)$ invariant if and only if
\begin{eqnarray*}
&&(F_1^*+F_2z)\Theta(z)=\Theta(z)(G_1+G_2^*z) \text{ and}
\\
&&(F_2^*+F_1z)\Theta(z)=\Theta(z)(G_2+G_1^*z) \text{ for all }z \in \mathbb{D},
\end{eqnarray*}
for some unique $G_1,G_2 \in \mathcal{B}(\mathcal{E}_{*})$, where $(\mathcal{E}_{*},\mathcal{E},\Theta)$ is the Beurling-Lax-Halmos representation of $\mathcal{M}$.
\\
Moreover, if $(M_{F_1^*+F_2z},M_{F_2^*+F_1z},M_z)$ on $H^2_{\mathcal{E}}(\mathbb{D})$ is a tetrablock isometry then $(M_{G_1+G_2^*z},M_{G_2+G_1^*z},M_z)$ is also a tetrablock isometry on $H^2(\mathcal{E_{*}})$ .
\end{theorem}
\begin{proof}
So let us assume that $\{0\} \neq \mathcal{M} \subseteq H^2_{\mathcal{E}}(\mathbb{D})$ is a $(M_{F_1^*+F_2z},M_{F_2^*+F_1z},M_z)$ invariant subspace. Let $\mathcal{M}=M_\Theta H^2_{\mathcal{E_*}}(\mathbb{D})$ be the Beurling-Lax-Halmos representation of $\mathcal{M}$, where $(\mathcal{E}_{*},\mathcal{E},\Theta)$ is an inner analytic function and $\mathcal{E}_{*}$ is an auxiliary Hilbert space.
Since $\mathcal{M}$ is $M_{F_1^*+F_2z}$ and $M_{F_2^*+F_1z}$ invariant also, we have
\begin{eqnarray*}
&&M_{F_1^*+F_2z}M_\Theta H^2_{\mathcal{E_*}}(\mathbb{D}) \subseteq M_\Theta H^2_{\mathcal{E_*}}(\mathbb{D}) \text{ and }
\\
&&M_{F_2^*+F_1z}M_\Theta H^2_{\mathcal{E_*}}(\mathbb{D}) \subseteq M_\Theta H^2_{\mathcal{E_*}}(\mathbb{D}).
\end{eqnarray*}
Now let us define two operators $X$ and $Y$ on $H^2(\mathcal{E}_{*})$ by the following way:
\begin{eqnarray*}
&&M_{F_1^*+F_2z}M_\Theta= M_\Theta X \text{ and }
\\
&&M_{F_2^*+F_1z}M_\Theta= M_\Theta Y.
\end{eqnarray*}
That $X$ and $Y$ are well defined and unique, follows from the fact that $\Theta$ is an inner analytic function, hence $M_{\Theta}$ is an isometry, (see \cite{Nagy-Foias}, Proposition 2.2, chapter V).
 \begin{eqnarray*}
 M_{F_1^*+F_2z}M_\Theta= M_\Theta X &\Rightarrow& M_\Theta^*M_{F_1^*+F_2z}^*M_\Theta=X^* \; [\text{ since $M_{\Theta}$ is an isometry}]
 \\
 &\Rightarrow& M_z^*M_\Theta^*M_{F_1^*+F_2z}^*M_\Theta=M_z^*X^*
 \\
 &\Rightarrow& M_\Theta^*M_{F_1^*+F_2z}^*M_\Theta M_z^* = M_z^*X^* \; [\text{ since $M_z$ commutes with $M_{\Theta}$ and $M_{F_1^*+F_2z}$}]
 \\
 &\Rightarrow& X^* M_z^*=M_z^*X^*.
 \end{eqnarray*}
 Hence $X$ commutes with $M_{z}$. Similarly one can prove that $Y$ commutes with $M_z$. So $X=M_{\Phi}$ and $Y=M_{\Psi}$, for some $\Phi, \Psi \in H^{\infty}(\mathcal{B}(\mathcal{E}_{*}))$.
Therefore we have
\begin{eqnarray}\label{s1}
&&M_{F_1^*+F_2z}M_\Theta= M_\Theta M_{\Phi} \text{ and }
\\
\label{s2}
&&M_{F_2^*+F_1z}M_\Theta= M_\Theta M_{\Psi}.
\end{eqnarray}
Multiplying $M_\Theta^*$ from left of (\ref{s1}) and (\ref{s2}) and using the fact that $M_\Theta$ is an isometry, we get
\begin{eqnarray}\label{s1'}
&& M_\Theta^*M_{F_1^*+F_2z}M_\Theta=  M_{\Phi} \text{ and }
\\
\label{s2'}
&& M_\Theta^*M_{F_2^*+F_1z}M_\Theta=  M_{\Psi}.
\end{eqnarray}
Multiplying $M_z^*$ from left of (\ref{s1'}) we get, $M_{\Theta}^*M_{F_2^*+F_1z}^*M_\Theta=M_z^*M_{\Phi}$, here we have used the fact that $M_{\Theta}$ and $M_z$ commute. Hence $M_{\Psi} = M_\Theta^*M_\Theta M_{\Psi} = M_\Theta^* M_{F_2^*+F_1z}M_\Theta= M_{\Phi}^*M_z$. Similarly dealing with equation (\ref{s2'}), we get $M_{\Phi} = M_{\Psi}^*M_z$.
Considering the power series expression of $\Phi$ and $\Psi$ and using that $M_{\Phi} = M_{\Psi}^*M_z$ and $M_{\Psi} = M_{\Phi}^*M_z$, we get $\Phi$ and $\Psi$ to be of the form $\Phi(z)=G_1+G_2^*z$ and $\Psi(z)= G_2 + G_1^*z$ for some $G_1,G_2 \in \mathcal{B}(\mathcal{E}_{*})$.
Uniqueness of $G_1$ and $G_2$ follows from the fact that $X$ and $Y$ are unique. The converse part is trivial. Hence the proof of the first part of the theorem.
\\
Moreover, suppose that $(M_{F_1^*+F_2z},M_{F_2^*+F_1z},M_z)$ is a tetrablock isometry.
To show that $(M_{G_1+G_2^*z},M_{G_2+G_1^*z},M_z)$ is also a tetrablock isometry we first show that they commute with each other. Commutativity of $M_{G_1+G_2^*z}$ and
$M_{G_2+G_1^*z}$ with $M_z$ is clear. Now
\begin{eqnarray*}
&&M_{G_1+G_2^*z}M_{G_2+G_1^*z}
\\
&=& M_{\Theta}^*M_{F_1^*+F_2z}M_\Theta M_{\Theta}^*M_{F_2^*+F_1z}M_\Theta \; [\text{ using equation (\ref{s1'}) and (\ref{s2'})}]
\\
&=& M_{\Theta}^*M_{F_1^*+F_2z}M_{F_2^*+F_1z}M_\Theta \; [\text{by equation (\ref{s2}) and that $M_\Theta$ is an isometry.}]
\\
&=&
M_{\Theta}^*M_{F_2^*+F_1z}M_{F_1^*+F_2z}M_\Theta \;  [\text{ since $M_{F_1^*+F_2z}$ and $M_{F_2^*+F_1z}$ commute.}]
\\
&=&
M_{\Theta}^*M_{F_2^*+F_1z}M_\Theta M_\Theta^*M_{F_1^*+F_2z}M_\Theta \; [\text {by equation (\ref{s1}) and that $M_\Theta$ is an isometry,}]
\\
&=&
M_{G_2+G_1^*z}M_{G_1+G_2^*z}.
\end{eqnarray*}
Since $(M_{F_1^*+F_2z},M_{F_2^*+F_1z},M_z)$ is a tetrablock isometry, we have $||M_{F_2^*+F_1z}|| \leq 1$, by part(3) of Theorem 5.7 in \cite{sir's tetrablock paper}. From the operator equation
$M_{G_2+G_1^*z}=M_{\Theta}^*M_{F_2^*+F_1z}M_\Theta$ we get that $||M_{G_2+G_1^*z}|| \leq 1$. From the proof of the first part, we have that $M_{\Phi} = M_{\Psi}^*M_z $
Hence $(M_{G_1+G_2^*z},M_{G_2+G_1^*z},M_z)$ is a tetrablock isometry invoking part(3) of Theorem 5.7 in \cite{sir's tetrablock paper}. This completes the proof of the theorem.
\end{proof}
Now we use Lemma \ref{scor1} to prove the following result which is a consequence of Theorem \ref{sthm1}.
\begin{corollary}
Let $F_1,F_2$ and $G_1,G_2$ be fundamental operators of $(A,B,P)$ and $(A^*,B^*,P^*)$ respectively. Then the triple $(M_{G_1+G_2^*z},M_{G_2+G_1^*z},M_z)$ is a tetrablock isometry whenever $(M_{F_1^*+F_2z},M_{F_2^*+F_1z},M_z)$ is a tetrablock isometry, provided $P^*$ is pure, i.e., $P^n \to 0$ strongly as $n \to \infty$.
\end{corollary}
\begin{proof}
Note that while proving the last part of Theorem \ref{sthm1}, we used the fact that the multiplier $M_{\Theta}$ is an isometry. Since $P^*$ is pure, by virtu of Proposition 3.5 of chapter VI in \cite{Nagy-Foias}, we note that the multiplier $M_{\Theta_{P^*}}$ is an isometry. From Lemma \ref{scor1}, we have
\begin{eqnarray*}
&&(F_1^*+F_2z)\Theta_{P^*}(z)=\Theta_{P^*}(z)(G_1+G_2^*z)
\\
&&(F_2^*+F_1z)\Theta_{P^*}(z)=\Theta_{P^*}(z)(G_2+G_1^*z), \text{ for all }z \in \mathbb{D}.
\end{eqnarray*}
Invoking the last part of Theorem \ref{sthm1}, we get the result as stated.
\end{proof}

\section{functional model}
In this section we find a functional model of pure tetrablock contractions. We first need to recall the functional model of pure contractions from \cite{Nagy-Foias}.
\\
Recall from the introduction that the characteristic function $\Theta_P$ of a contraction $P$ on a Hilbert space $\mathcal{H}$ is defined by
$$
\Theta_P(z)=[-P+zD_{P^*}(I_\mathcal{H}-zP^*)^{-1}D_P]|_{\mathcal{D}_P}, \text{ for all $z \in \mathbb{D}$}.
$$
By virtue of the relation (\ref{Maa8}), we get that each $\Theta_P(z)$ is an operator from $\mathcal{D}_P$ into $\mathcal{D}_{P^*}$.
The characteristic function induces a multiplication operator $M_{\Theta_P}$ from $H^2(\mathbb{D})\otimes \mathcal{D}_{P}$ into $H^2(\mathbb{D})\otimes \mathcal{D}_{P^*}$, defined by
$$
M_{\Theta_P}f(z)=\Theta_P(z)f(z), \text{ for all $f \in H^2(\mathbb{D})\otimes \mathcal{D}_{P}$ and $z \in \mathbb{D}$}.
$$
Note that $M_{\Theta_P}(M_z \otimes I_{\mathcal{D}_P})=(M_z \otimes I_{\mathcal{D}_{P^*}})M_{\Theta_P}$. Let us define
$$
\mathcal{H}_P=(H^2(\mathbb{D})\otimes \mathcal{D}_{P^*}) \ominus M_{\Theta_P}(H^2(\mathbb{D}) \otimes \mathcal{D}_P).
$$
In \cite{Nagy-Foias}, Sz.-Nagy and Foias showed that every pure contraction $P$ defined on an abstract Hilbert space $\mathcal{H}$ is unitarily equivalent to the operator $P_{\mathcal{H}_P}(M_z \otimes I_{\mathcal{D}_{P^*}})_{\mathcal{H}_P}$ on the Hilbert space $\mathcal{H}_P$ defined above, where $P_{\mathcal{H}_P}$ is the projection of $H^2(\mathbb{D})\otimes \mathcal{D}_{P^*}$ onto $\mathcal{H}_P$.
\\
Before going to state and proof the main result of this section, let us mention an interesting and well-known result, a proof of which  can be found in \cite{ch. function}. There it is proved for a commuting contractive d-tuple, for d $\geq 1$(Lemma 3.6). We shall write the proof here for the sake of completeness.
\\
Define $W: \mathcal{H} \to H^2(\mathbb{D})\otimes \mathcal{D}_{P^*} $ by
$$
W(h)=\sum_{n=0}^{\infty} z^n \otimes D_{P^*}{P^*}^nh, \text{ for all $h \in \mathcal{H}$}.
$$
It is easy to check that $W$ is an isometry for pure contractions $P$ and its adjoint is given by
$$
W^*(z^n \otimes \xi) = P^nD_{P^*} \xi, \text{ for all $\xi \in \mathcal{D}_{P^*}$ and $n \geq 0$.}
$$
\begin{lemma}\label{L0}
For every contraction $P$, the identity
\begin{eqnarray}\label{L1}
WW^*+M_{\Theta_P}M_{\Theta_P}^*=I_{H^2(\mathbb{D})\otimes \mathcal{D}_{P^*}}
\end{eqnarray}
holds.
\end{lemma}
\begin{proof}
As observed by Arveson in the proof of Theorem 1.2 in \cite{Arveson}, the operator $W^*$ satisfies the identity
$$
W^*(k_z \otimes \xi)= (I - \bar{z}P)^{-1}D_{P^*}\xi \text{ for $z \in \mathbb{D}$ and $\xi \in \mathcal{D}_{P^*}$},
$$
where $k_z(w):=(1-\langle w,z\rangle)^{-1}$ for all $w \in \mathbb{D}$. Therefore we have
\begin{eqnarray*}
&&\langle (WW^*+M_{\Theta_P}M_{\Theta_P}^*)(k_z \otimes \xi), (k_w \otimes \eta) \rangle
\\
&=& \langle W^*(k_z \otimes \xi),W^*(k_w \otimes \eta) \rangle + \langle M_{\Theta_P}^*(k_z \otimes \xi), M_{\Theta_P}^*(k_w \otimes \eta)  \rangle
\\
&=& \langle (I - \bar{z}P)^{-1}D_{P^*}\xi, (I - \bar{w}P)^{-1}D_{P^*}\eta  \rangle + \langle k_z \otimes \Theta_P(z)^*\xi, k_w \otimes \Theta_P(w)^*\eta  \rangle
\\
&=& \langle D_{P^*}(I - wP^*)^{-1}(I - \bar{z}P)^{-1}D_{P^*}\xi, \eta \rangle + \langle k_z, k_w \rangle \langle \Theta_P(w)\Theta_P(z)^*\xi, \eta  \rangle
\\
&=& \langle k_z \otimes \xi, k_w \otimes \eta \rangle \text{ for all $z ,w \in \mathbb{D}$ and $\xi, \eta \in \mathcal{D}_{P^*}$}.
\end{eqnarray*}
Where the last equality follows from the following well-known identity
$$
I - \Theta_P(w)\Theta_P(z)^* = (1 - w\bar{z})D_{P^*}(I - wP^*)^{-1}(I - \bar{z}P)^{-1}D_{P^*}.
$$
Now using the fact that $\{k_z: z \in \mathbb{D}\}$ forms a total set of $H^2(\mathbb{D})$, the assertion follows.
\end{proof}
The following theorem is the main result of this section.
\begin{theorem}\label{fm}
Let $(A,B,P)$ be a pure tetrablock contraction on a Hilbert space $\mathcal{H}$. Then the operators $A,B$ and $P$ are unitarily equivalent to $P_{\mathcal{H}_P}(I \otimes G_1^* + M_z \otimes G_2)|_{\mathcal{H}_P}, P_{\mathcal{H}_P}(I \otimes G_2^* + M_z \otimes G_1)|_{\mathcal{H}_P}$ and $P_{\mathcal{H}_P}(M_z \otimes I_{\mathcal{D}_{P^*}})|_{\mathcal{H}_P}$ respectively, where $G_1,G_2$ are the fundamental operators of $(A^*,B^*,P^*)$.
\end{theorem}
\begin{proof}
Since $W$ is an isometry, $WW^*$ is the projection onto $RanW$ and since $P$ is pure, $M_{\Theta_P}$ is also an isometry. So by Lemma \ref{L0}, we have $W(\mathcal{H}_P)=(H^2(\mathbb{D})\otimes \mathcal{D}_{P^*}) \ominus M_{\Theta_P}(H^2(\mathbb{D}) \otimes \mathcal{D}_P)$.
\begin{eqnarray*}
&&W^*(I \otimes G_1^* + M_z \otimes G_2)(z^n \otimes \xi)
\\
&=&
W^*(z^n \otimes G_1^* \xi) + W^*(z^{n+1} \otimes G_2 \xi)
\\
&=&
P^nD_{P^*}G_1^* \xi + P^{n+1}D_{P^*}G_2 \xi
\\
&=&
P^n(D_{P^*}G_1^*+PD_{P^*}G_2) \xi
\\
&=&
P^nAD_{P^*} \xi \;\;[\text{ by Lemma \ref{tetra}}]
\\
&=& AP^nD_{P^*} \xi = A W^*(z^n \otimes \xi).
\end{eqnarray*}
Therefore we have $W^*(I \otimes G_1^* + M_z \otimes G_2)=A W^*$ on vectors of the form $z^n \otimes \xi$, for all $n \geq 0$ and $\xi \in \mathcal{D}_{P^*}$, which span $H^2(\mathbb{D})\otimes \mathcal{D}_{P^*}$ and hence we have $W^*(I \otimes G_1^* + M_z \otimes G_2)=A W^*$, which implies $W^*(I \otimes G_1^* + M_z \otimes G_2)W=A$. Therefore $A$ is unitarily equivalent to $P_{\mathcal{H}_P}(I \otimes G_1^* + M_z \otimes G_2)|_{\mathcal{H}_P}$.
Also
\begin{eqnarray*}
&&W^*(I \otimes G_2^* + M_z \otimes G_1)(z^n \otimes \xi)
\\
&=&
W^*(z^n \otimes G_2^* \xi) + W^*(z^{n+1} \otimes G_1 \xi)
\\
&=&
P^nD_{P^*}G_2^* \xi + P^{n+1}D_{P^*}G_1 \xi
\\
&=&
P^n(D_{P^*}G_2^*+PD_{P^*}G_1) \xi
\\
&=&
P^nBD_{P^*} \xi \;\;[\text{ by Lemma \ref{tetra}}]
\\
&=& BP^nD_{P^*} \xi = B W^*(z^n \otimes \xi).
\end{eqnarray*}
Hence $W^*(I \otimes G_2^* + M_z \otimes G_1)=B W^*$ on vectors of the form $z^n \otimes \xi$, for all $n \geq 0$ and $\xi \in \mathcal{D}_{P^*}$. So, for the same reason we have $W^*(I \otimes G_2^* + M_z \otimes G_1)=B W^*$. Therefore $B$ is unitarily equivalent to $P_{\mathcal{H}_P}(I \otimes G_2^* + M_z \otimes G_1)|_{\mathcal{H}_P}$. And
\begin{eqnarray*}
W^*(M_z \otimes I)(z^n \otimes \xi)=W^*(z^{n+1} \otimes \xi)= P^{n+1}D_{P^*}\xi= PW^*(z^n \otimes \xi).
\end{eqnarray*}
Therefore by the same argument as above, we get $P$ is unitarily equivalent to $P_{\mathcal{H}_P}(M_z \otimes I_{\mathcal{D}_{P^*}})|_{\mathcal{H}_P}$. Note that the unitary operator which unitarizes $A,B$ and $P$ to their model operators is $W: \mathcal{H} \to \mathcal{H}_P$.
\\
This completes the proof of the theorem.
\end{proof}
We end this section with an important result which gives a functional model for a special class of tetrablock contractions, viz., pure tetrablock isometries. This is a consequence of Theorem \ref{fm}. This is important because this gives a relation between the fundamental operators $G_1$ and $G_2$ of adjoint of a pure tetrablock isometry.
\begin{corollary}\label{lastcor}
Let $(A,B,P)$ be a pure tetrablock isometry. Then $(A,B,P)$ is unitarily equivalent to $(M_{G_1^*+G_2z}, M_{G_2^*+G_1z},M_z)$, where $G_1$ and $G_2$ are the fundamental operators of $(A^*,B^*,P^*)$. Moreover, $G_1$ and $G_2$ satisfy equation (\ref{Maa14}).
\end{corollary}
\begin{proof}
Note that for an isometry $P$, the defect space $\mathcal{D}_P$ is zero, hence the charateristic function $\Theta_P$ is also zero. So for an isometry $P$, the space $\mathcal{H}_P$ becomes $H^2(\mathbb{D})\otimes \mathcal{D}_{P^*}$. So by Theorem \ref{fm}, we have the result. From the commutativity of the triple $(M_{G_1^*+G_2z}, M_{G_2^*+G_1z},M_z)$, it follows that $G_1$ and $G_2$ satisfy equation (\ref{Maa14}).
\end{proof}
\begin{remark}
In \cite{sir's tetrablock paper} (Theorem 5.10), it is showed that every pure tetrablock isometry $(A,B,P)$ on $\mathcal{H}$ is unitarily equivalent to $(M_{\tau_1^*+\tau_2z}, M_{\tau_2^*+\tau_1z},M_z)$ on $H^2_{\mathcal{E}}(\mathbb{D})$ for some $\tau_1, \tau_2 \in \mathcal{B}(\mathcal{E})$. Corollary \ref{lastcor} shows that the space $\mathcal{E}$ can be taken to $\mathcal{D}_{P^*}$ and the operators $\tau_1, \tau_2$ can be taken to be the fundamental operators of $(A^*,B^*,P^*)$.
\end{remark}

\section{A complete set of unitary invariant}
Given two contractions $P$ and $P'$ on Hilbert spaces $\mathcal{H}$ and $\mathcal{H'}$ respectively, we say that the characteristic functions of $P$ and $P'$ coincide if there are unitary operators $u: \mathcal{D}_P \to \mathcal{D}_{P'}$ and $u_{*}: \mathcal{D}_{P^*} \to \mathcal{D}_{{P'}^*}$ such that the following diagram commutes for all $z \in \mathbb{D}$,
$$
\begin{CD}
\mathcal{D}_P @>\Theta_P(z)>> \mathcal{D}_{P^*}\\
@Vu VV @VVu_{*} V\\
\mathcal{D}_{P'} @>>\Theta_{{P'}}(z)> \mathcal{D}_{{P'}^*}
\end{CD}.
$$
In \cite{Nagy-Foias}, Sz.-Nagy and Foias proved that the characteristic function of a c.n.u. contraction is a complete unitary invariant. In other words,
\begin{theorem}\label{nf}
Two completely non-unitary contractions are unitarily equivalent if and only
if their characteristic functions coincide.
\end{theorem}
In this section, we give a complete set of unitary invariant for a pure tetrablock contraction.
\begin{proposition}
If two tetrablock contractions $(A,B,P)$ and $(A',B',P')$ defined on $\mathcal{H}$ and $\mathcal{H'}$ respectively are unitarily equivalent then so are their fundamental operators.
\end{proposition}
\begin{proof}
Let $U:\mathcal{H} \to \mathcal{H'}$ be a unitary such that $UA=A'U, UB=B'U$ and $UP=P'U$.
Then we have
$$
UD_P^2=U(I-P^*P)=U-{P'}^*PU=D_{P'}^2U,
$$
which gives $UD_P=D_{P'}U$. Let $\tilde{U}=U|_{\mathcal{D}_P}$. Then note that $\tilde{U}\in \mathcal{B}(\mathcal{D}_P, \mathcal{D}_{P'})$ and $\tilde{U}D_P=D_{P'}\tilde{U}$. Let $F_1,F_2$ and $F_1',F_2'$ be fundamental operators of $(A,B,P)$ and $(A',B',P')$ respectively. Then
\begin{eqnarray*}
D_{P'}\tilde{U}F_1\tilde{U}^*D_{P'}=\tilde{U}D_PF_1D_P\tilde{U}^*=\tilde{U}(A-B^*P)D_P{\tilde{U}}^*=A'-B'^*P'=D_{P'}F_1'D_{P'}.
\end{eqnarray*}
Therefore we have $\tilde{U}F_1\tilde{U}^*=F_1'$. Similarly one can prove that $\tilde{U}F_2\tilde{U}^*=F_2'$.
\\
This completes the proof.
\end{proof}
The next result is a sort of converse to the previous proposition for pure tetrablock contractions.
\begin{proposition}
Let $(A,B,P)$ and $(A',B',P')$ be two pure tetrablock contractions defined on $\mathcal{H}$ and $\mathcal{H'}$ respectively. Suppose that the characteristic functions of $P$ and $P'$ coincide and the fundamental operators $(G_1,G_2)$ of $(A^*,B^*,P^*)$ and $(G_1',G_2')$ of $(A'^*,B'^*,P'^*)$ are unitarily equivalent by the same unitary that is involved in the coincidence of the characteristic functions of $P$ and $P'$. Then $(A,B,P)$ and $(A',B',P')$ are unitarily equivalent.
\end{proposition}
\begin{proof}
Let $u: \mathcal{D}_P \to \mathcal{D}_{P'}$ and $u_{*}: \mathcal{D}_{P^*} \to \mathcal{D}_{{P'}^*}$ be unitary operators such that
$$
u_{*}G_1=G_1'u_{*}, \; u_{*}G_2=G_2'u_{*} \text{ and } u_{*} \Theta_P(z) = \Theta_{P'}(z) u \text{ holds for all $z \in \mathbb{D}$}.
$$
The unitary operator $u_{*}: \mathcal{D}_{P^*} \to \mathcal{D}_{{P'}^*}$ induces a unitary operator $U_{*}: H^2(\mathbb{D})\otimes \mathcal{D}_{P^*} \to H^2(\mathbb{D})\otimes \mathcal{D}_{P'^*}$ defined by $U_{*}(z^n \otimes \xi)=(z^n \otimes u_{*}\xi)$ for all $\xi \in \mathcal{D}_{P^*}$ and $n \geq 0$.
Note that
$$
U_{*}(M_{\Theta_P}f(z))=u_{*}\Theta_{P}(z)f(z)=\Theta_{P'}(z)uf(z)=M_{\Theta_{P'}}(uf(z)),
$$
for all $f \in H^2(\mathbb{D})\otimes \mathcal{D}_{P}$ and $z \in \mathbb{D}$. Hence $U_{*}$ takes $RanM_{\Theta_P}$ onto $RanM_{\Theta_{P'}}$. Since $U_{*}$ is unitary, we have
$$
U_{*}(\mathcal{H}_P) = U_{*}((RanM_{\Theta_P})^\bot)=(U_{*}RanM_{\Theta_P})^{\bot}=(RanM_{\Theta_{P'}})^\bot=\mathcal{H}_{P'}.
$$
By definition of $U_{*}$ we have
\begin{eqnarray*}
U_{*}(I \otimes G_1^* + M_z \otimes G_2)^*&=&(I \otimes u_{*})(I \otimes G_1 + M_z^* \otimes G_2^*)
\\
&=&
I \otimes u_{*} G_1 + M_z^* \otimes u_{*} G_2^*
\\
&=&
I \otimes G_1'u_{*} + M_z^* \otimes G_2'^*u_{*}
\\
&=&
(I \otimes G_1' + M_z^* \otimes G_2'^*)(I \otimes u_{*})=(I \otimes G_1'^* + M_z \otimes G_2')^*U_{*}.
\end{eqnarray*}
Similar calculation gives us
$$U_{*}(I \otimes G_2^* + M_z \otimes G_1)^*=(I \otimes G_2'^* + M_z \otimes G_1')^*U_{*}.$$
Therefore $\mathcal{H}_{P'}=U_{*}(\mathcal{H}_P)$ is a co-invariant subspace of $(I \otimes G_1'^* + M_z \otimes G_2')$ and $(I \otimes G_2'^* + M_z \otimes G_1')$.
Hence
$$
P_{\mathcal{H}_P}(I \otimes G_1^* + M_z \otimes G_2)|_{\mathcal{H}_P} \cong P_{\mathcal{H}_{P'}}(I \otimes G_1'^* + M_z \otimes G_2')|_{\mathcal{H}_{P'}}
$$
and
$$
P_{\mathcal{H}_P}(I \otimes G_2^* + M_z \otimes G_1)|_{\mathcal{H}_P} \cong P_{\mathcal{H}_{P'}}(I \otimes G_2'^* + M_z \otimes G_1')|_{\mathcal{H}_{P'}}
$$
and the unitary operator which unitarizes them is $U_{*}|_{\mathcal{H}_P}:\mathcal{H}_P \to \mathcal{H}_{P'}$.
\\
And also by definition of $U_{*}$ we have
$$
U_{*}(M_z \otimes I_{\mathcal{D}_{P^*}})=(I \otimes u_{*})(M_z \otimes I_{\mathcal{D}_{P^*}})=(M_z \otimes I_{\mathcal{D}_{P'^*}})(I \otimes u_{*})=(M_z \otimes I_{\mathcal{D}_{P'^*}})U_{*}.
$$
So $P_{\mathcal{H}_P}(M_z \otimes I_{\mathcal{D}_{P^*}})|_{\mathcal{H}_P} \cong P_{\mathcal{H}_{P'}}(M_z \otimes I_{\mathcal{D}_{P'^*}})|_{\mathcal{H}_{P'}}$ and the same unitary $U_{*}|_{\mathcal{H}_P}:\mathcal{H}_P \to \mathcal{H}_{P'}$ unitarizes them. Therefore $(A,B,P) \cong (A',B',P')$.
\\
This completes the proof of the theorem.

\end{proof}
Combining the last two propositions and Theorem \ref{nf}, we get the following theorem which is the main result of this section.
\begin{theorem}\label{unitary inv}
 Let $(A,B,P)$ and $(A',B',P')$ be two pure tetrablock contractions defined on $\mathcal{H}$ and $\mathcal{H'}$ respectively. Suppose $(G_1,G_2)$ and $(G_1',G_2')$ are fundamental operators of $(A^*,B^*,P^*)$ and $(A'^*,B'^*,P'^*)$ respectively. Then $(A,B,P)$ is unitarily equivalent to $(A',B',P')$ if and only if the characteristic functions of $P$ and $P'$ coincide and $(G_1,G_2)$ is unitarily equivalent to $(G_1',G_2')$ by the same unitary that is involved in the coincidence of the characteristic functions of $P$ and $P'$.
\end{theorem}

\section{an example}
\subsection{Fundamental operators}
Consider the Hilbert space
$$
H^2(\mathbb{D}^2)=\{ f: \mathbb{D}^2 \to \mathbb{C}: f(z_1,z_2)=\sum_{i=0}^\infty\sum_{j=0}^\infty a_{ij}z_1^iz_2^j \text{ with } \sum_{i=0}^\infty\sum_{j=0}^\infty |a_{ij}|^2 < \infty \}
$$ with the inner product $\langle \sum_{i=0}^\infty\sum_{j=0}^\infty a_{ij}z_1^iz_2^j, \sum_{i=0}^\infty\sum_{j=0}^\infty b_{ij}z_1^iz_2^j \rangle
 = \sum_{i=0}^\infty\sum_{j=0}^\infty a_{ij}\bar{b_{ij}} $. Consider the commuting triple of operators $(M_{z_1},M_{z_2},M_{z_1z_2})$ on $H^2(\mathbb{D}^2)$. It can be easily checked with the help of Theorem 5.4 in \cite{sir's tetrablock paper}, that $(M_{z_1},M_{z_2},M_{z_1z_2})$ is a tetrablock unitary on $L^2(\mathbb{T}^2)$. Note that $(M_{z_1},M_{z_2},M_{z_1z_2})$ on $H^2(\mathbb{D}^2)$ is the restriction of the tetrablock unitary $(M_{z_1},M_{z_2},M_{z_1z_2})$ on $L^2(\mathbb{T}^2)$ to the common invariant subspace $H^2(\mathbb{D}^2)$ (naturally embedded) of $L^2(\mathbb{T}^2)$. Hence by definition, $(M_{z_1},M_{z_2},M_{z_1z_2})$ on $H^2(\mathbb{D}^2)$ is a tetrablock isometry. In this section, we calculate fundamental operators of the tetrablock co-isometry $(M_{z_1}^*,M_{z_2}^*,M_{z_1z_2}^*)$ on $H^2(\mathbb{D}^2)$. For notational convenience, we denote $M_{z_1}$,$M_{z_2}$ and $M_{z_1z_2}$ by $A,B$ and $P$ respectively.

Note that every element $f\in H^2(\mathbb{D}^2)$ has the form $f(z_1,z_2)=\sum_{i=0}^\infty\sum_{j=0}^\infty a_{ij}z_1^iz_2^j$ where $a_{ij}\in \mathbb{C},$ for all $i,j \geq 0$. So we can write $f$ in the matrix form
  $$ \left( \left( a_{ij} \right) \right)_{i,j=0}^\infty =
\begin{pmatrix}
a_{00} & a_{01} & a_{02} & \dots \\
a_{10} & a_{11} & a_{12} &\dots  \\
a_{20} & a_{21} & a_{22} & \dots \\
\vdots & \vdots & \vdots & \ddots
\end{pmatrix},
$$
where $(ij)$-th entry in the matrix, denotes the coefficient of $z_1^iz_2^j$ in $f(z_1,z_2)=\sum_{i=0}^\infty\sum_{j=0}^\infty a_{ij}z_1^iz_2^j.$
We shall write the matrix form instead of writing the series. In this notation,
 \begin{eqnarray}
 &&A( \; \left( \left( a_{ij} \right) \right)_{i,j=0}^\infty \; ) = \left( a_{(i-1)j} \right),\; B( \; \left( \left( a_{ij} \right) \right)_{i,j=0}^\infty \; ) = \left( a_{i(j-1)} \right)
 \\
 &&\mbox{ and } P( \; \left( \left( a_{ij} \right) \right)_{i,j=0}^\infty \; ) = \left( a_{(i-1)(j-1)} \right) \label{ABandP},
 \end{eqnarray}
 with the convention that $a_{ij}$ is zero if either $i$ or $j$ is negative.
 \begin{lemma}\label{ABP^*}
 The adjoints of the operators $A,B$ and $P$ are as follows.
 \begin{eqnarray*}
 &&A^*( \; \left( \left( a_{ij} \right) \right)_{i,j=0}^\infty \; ) = \left( a_{(i+1)j} \right),\; B^*( \; \left( \left( a_{ij} \right) \right)_{i,j=0}^\infty \; ) = \left( a_{i(j+1)} \right)
 \\
 &&\mbox{ and } P^*( \; \left( \left( a_{ij} \right) \right)_{i,j=0}^\infty \; ) = \left( a_{(i+1)(j+1)} \right).
 \end{eqnarray*}
 \end{lemma}
 \begin{proof}
 This is a matter of easy inner product calculations.
 \end{proof}
 \begin{lemma}\label{def space}
The defect space of $P^*$ in the matrix form is
$$\mathcal{D}_{P^*}
=
\{ \begin{pmatrix}
a_{00} & a_{01} & a_{02} & \dots \\
a_{10} & 0 & 0 &\dots  \\
a_{20} & 0 & 0 & \dots \\
\vdots & \vdots & \vdots & \ddots
\end{pmatrix}:  |a_{00}|^2 + \sum_{j=1}^\infty|a_{0j}|^2 + \sum_{j=1}^\infty|a_{j0}|^2 <
\infty\}.$$ The defect space in the function form is
$\overline{span}\{ 1, z_1^i,z_2^j: i,j \geq 1 \}$. The defect
operator for $P^*$ is $$D_{P^*}
\begin{pmatrix}
a_{00} & a_{01} & a_{02} & \dots \\
a_{10} & a_{11} & a_{12} &\dots  \\
a_{20} & a_{21} & a_{22} & \dots \\
\vdots & \vdots & \vdots & \ddots
\end{pmatrix}
=
\begin{pmatrix}
a_{00} & a_{01} & a_{02} & \dots \\
a_{10} & 0 & 0 &\dots  \\
a_{20} & 0 & 0 & \dots \\
\vdots & \vdots & \vdots & \ddots
\end{pmatrix}.
$$
\end{lemma}
\begin{proof}
Since $P$ is an isometry, $D_{P^*}$ is a projection onto $Range(P)^{\perp} = H^2(\mathbb{D}^2) \ominus Range(P)$. The rest follows from the formula for $P$ in (\ref{ABandP}).
\end{proof}
\begin{definition}\label{funda}
Define $G_1,G_2:\mathcal{D}_{P^*} \to \mathcal{D}_{P^*}$ by
\begin{eqnarray*} \label{B}
G_1\begin{pmatrix}
a_{00} & a_{01} & a_{02} & \dots \\
a_{10} & 0 & 0 &\dots  \\
a_{20} & 0 & 0 & \dots \\
\vdots & \vdots & \vdots & \ddots
\end{pmatrix}
=
\begin{pmatrix}
a_{10} & 0 & 0 & \dots \\
a_{20} & 0 & 0 &\dots  \\
a_{30} & 0 & 0 & \dots \\
\vdots & \vdots & \vdots & \ddots
\end{pmatrix} \text{ and}
\end{eqnarray*}
\begin{eqnarray*} \label{B}
G_2\begin{pmatrix}
a_{00} & a_{01} & a_{02} & \dots \\
a_{10} & 0 & 0 &\dots  \\
a_{20} & 0 & 0 & \dots \\
\vdots & \vdots & \vdots & \ddots
\end{pmatrix}
=
\begin{pmatrix}
a_{01} & a_{02} & a_{03} &\dots  \\
0 & 0 & 0 & \dots \\
0 & 0 & 0 & \dots \\
\vdots & \vdots & \vdots & \ddots
\end{pmatrix},
\end{eqnarray*}
for all $a_{j0}, a_{0j} \in \mathbb{C}, j=0,1,2,\dots \text{ with } |a_{00}|^2 + \sum_{j=1}^\infty|a_{0j}|^2 + \sum_{j=1}^\infty|a_{j0}|^2 < \infty$.
\end{definition}
\begin{lemma}
The operators $G_1$ and $G_2$ as defined in Definition \ref{funda}, is the fundamental operators of $(A^*,B^*,P^*)$.
\end{lemma}
\begin{proof}
To show that $G_1$ and $G_2$ is the fundamental operators of $(A^*,B^*,P^*)$, we have to show that $G_1$ and $G_2$ satisfy the fundamental equation $A^*-BP^*=D_{P^*}G_1D_{P^*}$ and $B^*-AP^*=D_{P^*}G_2D_{P^*}$. Using Lemma \ref{ABP^*}, we get
\begin{eqnarray*}
&&(A^*-BP^*)
\begin{pmatrix}
a_{00} & a_{01} & a_{02} & \dots \\
a_{10} & a_{11} & a_{12} &\dots  \\
a_{20} & a_{21} & a_{22} & \dots \\
\vdots & \vdots & \vdots & \ddots
\end{pmatrix}
\\
&=&
\begin{pmatrix}
a_{10} & a_{11} & a_{12} &\dots  \\
a_{20} & a_{21} & a_{22} & \dots \\
a_{30} & a_{31} & a_{32} & \dots \\
\vdots & \vdots & \vdots & \ddots
\end{pmatrix}
-
B
\begin{pmatrix}
a_{11} & a_{12} & a_{13} & \dots \\
a_{21} & a_{22} & a_{23} &\dots  \\
a_{31} & a_{32} & a_{33} & \dots \\
\vdots & \vdots & \vdots & \ddots
\end{pmatrix}
\\
&=&
\begin{pmatrix}
a_{10} & a_{11} & a_{12} &\dots  \\
a_{20} & a_{21} & a_{22} & \dots \\
a_{30} & a_{31} & a_{32} & \dots \\
\vdots & \vdots & \vdots & \ddots
\end{pmatrix}
-
\begin{pmatrix}
0 & a_{11} & a_{12} &  \dots \\
0 & a_{21} & a_{22} &  \dots  \\
0 & a_{31} & a_{32} &  \dots \\
\vdots & \vdots & \vdots & \ddots
\end{pmatrix}
=
\begin{pmatrix}
a_{10} & 0 & 0 &\dots  \\
a_{20} & 0 & 0 & \dots \\
a_{30} & 0 & 0 & \dots \\
\vdots & \vdots & \vdots & \ddots
\end{pmatrix}
\end{eqnarray*}
and
\begin{eqnarray*}
&&(B^*-AP^*)
\begin{pmatrix}
a_{00} & a_{01} & a_{02} & \dots \\
a_{10} & a_{11} & a_{12} &\dots  \\
a_{20} & a_{21} & a_{22} & \dots \\
\vdots & \vdots & \vdots & \ddots
\end{pmatrix}
\\
&=&
\begin{pmatrix}
a_{01} & a_{02} & a_{03} & \dots  \\
a_{11} & a_{12} & a_{13} & \dots \\
a_{21} & a_{22} & a_{23} & \dots \\
\vdots & \vdots & \vdots & \ddots
\end{pmatrix}
-
A
\begin{pmatrix}
a_{11} & a_{12} & a_{13} & \dots \\
a_{21} & a_{22} & a_{23} & \dots  \\
a_{31} & a_{32} & a_{33} & \dots \\
\vdots & \vdots & \vdots & \ddots
\end{pmatrix}
\\
&=&
\begin{pmatrix}
a_{01} & a_{02} & a_{03} & \dots  \\
a_{11} & a_{12} & a_{13} & \dots \\
a_{21} & a_{22} & a_{23} & \dots \\
\vdots & \vdots & \vdots & \ddots
\end{pmatrix}
-
\begin{pmatrix}
0  &  0 & 0 & \dots \\
a_{11} & a_{12} & a_{13} & \dots \\
a_{21} & a_{22} & a_{23} &\dots  \\
\vdots & \vdots & \vdots & \ddots
\end{pmatrix}
=
\begin{pmatrix}
a_{01} & a_{02} & a_{03} &\dots  \\
0 & 0 & 0 & \dots \\
0 & 0 & 0 & \dots \\
\vdots & \vdots & \vdots & \ddots
\end{pmatrix}.
\end{eqnarray*}
Using Lemma \ref{def space} and Definition \ref{funda}, we get
\begin{eqnarray*}
D_{P^*}G_1D_{P^*}
\begin{pmatrix}
a_{00} & a_{01} & a_{02} & \dots \\
a_{10} & a_{11} & a_{12} &\dots  \\
a_{20} & a_{21} & a_{22} & \dots \\
\vdots & \vdots & \vdots & \ddots
\end{pmatrix} =
D_{P^*}G_1
\begin{pmatrix}
a_{00} & a_{01} & a_{02} & \dots \\
a_{10} & 0 & 0 &\dots  \\
a_{20} & 0 & 0 & \dots \\
\vdots & \vdots & \vdots & \ddots
\end{pmatrix}
=
\begin{pmatrix}
a_{10} & 0 & 0 & \dots \\
a_{20} & 0 & 0 &\dots  \\
a_{30} & 0 & 0 & \dots \\
\vdots & \vdots & \vdots & \ddots
\end{pmatrix} \mbox{ and }
\\
D_{P^*}G_2D_{P^*}
\begin{pmatrix}
a_{00} & a_{01} & a_{02} & \dots \\
a_{10} & a_{11} & a_{12} &\dots  \\
a_{20} & a_{21} & a_{22} & \dots \\
\vdots & \vdots & \vdots & \ddots
\end{pmatrix} =
D_{P^*}G_2
\begin{pmatrix}
a_{00} & a_{01} & a_{02} & \dots \\
a_{10} & 0 & 0 &\dots  \\
a_{20} & 0 & 0 & \dots \\
\vdots & \vdots & \vdots & \ddots
\end{pmatrix}
=
\begin{pmatrix}
a_{01} & a_{02} & a_{03} &\dots  \\
0 & 0 & 0 & \dots \\
0 & 0 & 0 & \dots \\
\vdots & \vdots & \vdots & \ddots
\end{pmatrix}.
\end{eqnarray*}
Therefore, $G_1$ and $G_2$ are fundamental operators of $(A^*,B^*,P^*)$.
\end{proof}
\subsection{Explicit unitary equivalence}
From Corollary \ref{lastcor}, we now know that if $(A,B,P)$ is a pure tetrablock isometry, then $(A,B,P)$ is unitarily equivalent to $(M_{G_1^*+G_2z}, M_{G_2^*+G_1z},M_z)$, where $G_1$ and $G_2$ are the fundamental operators of $(A^*,B^*,P^*)$. The operator $M_{z_1z_2}$ on $H^2(\mathbb{D}^2)$ is a pure contraction as can be checked from the formula of $P^*$ in Lemma \ref{ABP^*}. In the final theorem of this section, we find the unitary operator which implement the unitary equivalence of the pure tetrablock isometry $(A,B,P)$ on $H^2(\mathbb{D}^2)$.
\begin{theorem}
The operator $U: H^2(\mathbb{D}^2) \to H^2_{\mathcal{D}_{P^*}}(\mathbb{D})$ defined by
\begin{eqnarray}\label{uni}
Uf(z)=D_{P^*}(I - zP^*)^{-1}f, \text{ for all } f \in H^2(\mathbb{D}^2) \text{ and } z \in \mathbb{D}
\end{eqnarray}
is a unitary operator and satisfies $U^*(M_{G_1^*+G_2z}, M_{G_2^*+G_1z},M_z)U=(A,B,P)$.
\end{theorem}
\begin{proof}
We first prove that $U$ is one-one. Expanding the series in (\ref{uni}), we get
\begin{eqnarray}\label{sform}
Uf(z)=D_{P^*}f+zD_{P^*}P^*f+z^2D_{P^*}{P^*}^2f+ \cdots.
\end{eqnarray} Therefore
\begin{eqnarray*}
\lVert Uf \rVert^2_{H^2_{\mathcal{D}_{P^*}}(\mathbb{D})}&=&\lVert D_{P^*}f\rVert^2_{\mathcal{D}_{P^*}} + \lVert D_{P^*}P^*f\rVert^2_{\mathcal{D}_{P^*}} + \lVert D_{P^*}{P^*}^2f\rVert^2_{\mathcal{D}_{P^*}} +\cdots
\\
&=& \lVert f\rVert^2 -\lim_{n \to \infty} \lVert {P^*}^nf \rVert^2 = \lVert f \rVert^2_{H^2(\mathbb{D}^2)}.\;\; [\text{ since $P$ is pure.}]
\end{eqnarray*}
From the explicit series form of $U$ (equation \ref{sform}), we see that $U$ in matrix form of an element, is the following.
\begin{eqnarray*}
&&U
\begin{pmatrix}
a_{00} & a_{01} & a_{02} & \dots \\
a_{10} & a_{11} & a_{12} &\dots  \\
a_{20} & a_{21} & a_{22} & \dots \\
\vdots & \vdots & \vdots & \ddots
\end{pmatrix}
\\
&=&
\begin{pmatrix}
a_{00} & a_{01} & a_{02} & \dots \\
a_{10} & 0 & 0 &\dots  \\
a_{20} & 0 & 0 & \dots \\
\vdots & \vdots & \vdots & \ddots
\end{pmatrix}
+
z\begin{pmatrix}
a_{11} & a_{12} & a_{13} & \dots \\
a_{21} & 0 & 0 &\dots  \\
a_{31} & 0 & 0 & \dots \\
\vdots & \vdots & \vdots & \ddots
\end{pmatrix}
+
z^2
\begin{pmatrix}
a_{22} & a_{23} & a_{24} & \dots \\
a_{32} & 0 & 0 &\dots  \\
a_{42} & 0 & 0 & \dots \\
\vdots & \vdots & \vdots & \ddots
\end{pmatrix}
+
\cdots .
\end{eqnarray*} From this representation, it is easy to see that $U$ is onto $H^2_{\mathcal{D}_{P^*}}(\mathbb{D})$.
It can be easily checked by definition of $G_1$ and $G_2$ (Definition \ref{funda}) that
$$
G_1^*
\begin{pmatrix}
a_{00} & a_{01} & a_{02} & \dots \\
a_{10} & 0 & 0 &\dots  \\
a_{20} & 0 & 0 & \dots \\
\vdots & \vdots & \vdots & \ddots
\end{pmatrix}
=
\begin{pmatrix}
0 & 0 & 0 & 0 & \dots \\
a_{00} & 0 & 0 & 0 & \dots \\
a_{10} & 0 & 0 & 0 & \dots  \\
a_{20} & 0 & 0 &0 &  \dots \\
\vdots & \vdots & \vdots & \vdots &\ddots
\end{pmatrix}\text{ and}
$$
$$
G_2^*
\begin{pmatrix}
a_{00} & a_{01} & a_{02} & \dots \\
a_{10} & 0 & 0 &\dots  \\
a_{20} & 0 & 0 & \dots \\
\vdots & \vdots & \vdots & \ddots
\end{pmatrix}
=
\begin{pmatrix}
0 & a_{00} & a_{01} & a_{02} & \dots \\
0 & 0 & 0 & 0 &\dots  \\
0 & 0 & 0 & 0 & \dots \\
\vdots & \vdots & \vdots &\vdots & \ddots
\end{pmatrix}.
$$
To prove $U^*(M_{G_1^*+G_2z}, M_{G_2^*+G_1z},M_z)U=(A,B,P)$, we proceed by proving $U^*M_zU=P$ first.
\begin{eqnarray*}
&&U^*M_zU
\begin{pmatrix}
a_{00} & a_{01} & a_{02} & \dots \\
a_{10} & a_{11} & a_{12} &\dots  \\
a_{20} & a_{21} & a_{22} & \dots \\
\vdots & \vdots & \vdots & \ddots
\end{pmatrix}
\\
&=&
U^*
\left(
z\begin{pmatrix}
a_{00} & a_{01} & a_{02} & \dots \\
a_{10} & 0 & 0 &\dots  \\
a_{20} & 0 & 0 & \dots \\
\vdots & \vdots & \vdots & \ddots
\end{pmatrix}
+
z^2\begin{pmatrix}
a_{11} & a_{12} & a_{13} & \dots \\
a_{21} & 0 & 0 &\dots  \\
a_{31} & 0 & 0 & \dots \\
\vdots & \vdots & \vdots & \ddots
\end{pmatrix}
+
z^3\begin{pmatrix}
a_{22} & a_{23} & a_{24} & \dots \\
a_{32} & 0 & 0 &\dots  \\
a_{42} & 0 & 0 & \dots \\
\vdots & \vdots & \vdots & \ddots
\end{pmatrix}
+ \cdots
\right)
\\
&=&
\begin{pmatrix}
0 & 0 & 0 & 0 & \dots \\
0 & a_{00} & a_{01} & a_{02} & \dots \\
0 & a_{10} & a_{11} & a_{12} &\dots  \\
0 & a_{20} & a_{21} & a_{22} & \dots \\
\vdots & \vdots & \vdots & \vdots& \ddots
\end{pmatrix}
=
P
\begin{pmatrix}
a_{00} & a_{01} & a_{02} & \dots \\
a_{10} & a_{11} & a_{12} &\dots  \\
a_{20} & a_{21} & a_{22} & \dots \\
\vdots & \vdots & \vdots & \ddots
\end{pmatrix}.
\end{eqnarray*}
Now to prove $M_{z_1}=U^*M_{G_1^*+zG_2}U$, we first calculate $M_{G_1^*+zG_2}U$.
\begin{eqnarray*}
&& M_{G_1^*+zG_2}U
\begin{pmatrix}
a_{00} & a_{01} & a_{02} & \dots \\
a_{10} & a_{11} & a_{12} &\dots  \\
a_{20} & a_{21} & a_{22} & \dots \\
\vdots & \vdots & \vdots & \ddots
\end{pmatrix}
\\
&=&M_{G_1^*+zG_2}
\left(
\begin{pmatrix}
a_{00} & a_{01} & a_{02} & \dots \\
a_{10} & 0 & 0 &\dots  \\
a_{20} & 0 & 0 & \dots \\
\vdots & \vdots & \vdots & \ddots
\end{pmatrix}
+
z\begin{pmatrix}
a_{11} & a_{12} & a_{13} & \dots \\
a_{21} & 0 & 0 &\dots  \\
a_{31} & 0 & 0 & \dots \\
\vdots & \vdots & \vdots & \ddots
\end{pmatrix}
+
z^2
\begin{pmatrix}
a_{22} & a_{23} & a_{24} & \dots \\
a_{32} & 0 & 0 &\dots  \\
a_{42} & 0 & 0 & \dots \\
\vdots & \vdots & \vdots & \ddots
\end{pmatrix}
+
\cdots
\right)
\end{eqnarray*}
\begin{eqnarray*}
&=&
\left(
\begin{pmatrix}
0 & 0 & 0 & 0 & \dots \\
a_{00} & 0 & 0 & 0 & \dots \\
a_{10} & 0 & 0 & 0 & \dots  \\
a_{20} & 0 & 0 &0 &  \dots \\
\vdots & \vdots & \vdots & \vdots &\ddots
\end{pmatrix}
+
z
\begin{pmatrix}
0 & 0 & 0 & 0 & \dots \\
a_{11} & 0 & 0 & 0 & \dots \\
a_{21} & 0 & 0 & 0 & \dots  \\
a_{31} & 0 & 0 &0 &  \dots \\
\vdots & \vdots & \vdots & \vdots &\ddots
\end{pmatrix}
+
z^2
\begin{pmatrix}
0 & 0 & 0 & 0 & \dots \\
a_{22} & 0 & 0 & 0 & \dots \\
a_{32} & 0 & 0 & 0 & \dots  \\
a_{42} & 0 & 0 &0 &  \dots \\
\vdots & \vdots & \vdots & \vdots &\ddots
\end{pmatrix}
+ \cdots
\right)
+
\\
&&
\left(
z
\begin{pmatrix}
a_{01} & a_{02} & a_{03} &\dots  \\
0 & 0 & 0 & \dots \\
0 & 0 & 0 & \dots \\
\vdots & \vdots & \vdots & \ddots
\end{pmatrix}
+
z^2
\begin{pmatrix}
a_{12} & a_{13} & a_{14} &\dots  \\
0 & 0 & 0 & \dots \\
0 & 0 & 0 & \dots \\
\vdots & \vdots & \vdots & \ddots
\end{pmatrix}
+
z^3
\begin{pmatrix}
a_{23} & a_{24} & a_{25} &\dots  \\
0 & 0 & 0 & \dots \\
0 & 0 & 0 & \dots \\
\vdots & \vdots & \vdots & \ddots
\end{pmatrix}
+ \cdots
\right)
\\
&=&
\left(
\begin{pmatrix}
0 & 0 & 0 & 0 & \dots \\
a_{00} & 0 & 0 & 0 & \dots \\
a_{10} & 0 & 0 & 0 & \dots  \\
a_{20} & 0 & 0 &0 &  \dots \\
\vdots & \vdots & \vdots & \vdots &\ddots
\end{pmatrix} +
z
\begin{pmatrix}
a_{01} & a_{02} & a_{03} & a_{04} & \dots \\
a_{11} & 0 & 0 & 0 & \dots \\
a_{21} & 0 & 0 & 0 & \dots  \\
a_{31} & 0 & 0 &0 &  \dots \\
\vdots & \vdots & \vdots & \vdots &\ddots
\end{pmatrix}
+
z^2
\begin{pmatrix}
 a_{12} & a_{13} & a_{14} & a_{15} & \dots \\
 a_{22} & 0 & 0 & 0 &\dots  \\
 a_{32} & 0 & 0 & 0 &\dots \\
 a_{42} & 0 & 0 & 0 & \dots \\
 \vdots & \vdots &\vdots &\vdots & \ddots
\end{pmatrix}
+ \cdots
\right).
\end{eqnarray*}
Therefore
\begin{eqnarray*}
 &&U^*M_{G_1^*+zG_2}U
\begin{pmatrix}
a_{00} & a_{01} & a_{02} & \dots \\
a_{10} & a_{11} & a_{12} &\dots  \\
a_{20} & a_{21} & a_{22} & \dots \\
\vdots & \vdots & \vdots & \ddots
\end{pmatrix}=
\\
&&
U^*
\left(
\begin{pmatrix}
0 & 0 & 0 & 0 & \dots \\
a_{00} & 0 & 0 & 0 & \dots \\
a_{10} & 0 & 0 & 0 & \dots  \\
a_{20} & 0 & 0 &0 &  \dots \\
\vdots & \vdots & \vdots & \vdots &\ddots
\end{pmatrix} +
z
\begin{pmatrix}
a_{01} & a_{02} & a_{03} & a_{04} & \dots \\
a_{11} & 0 & 0 & 0 & \dots \\
a_{21} & 0 & 0 & 0 & \dots  \\
a_{31} & 0 & 0 &0 &  \dots \\
\vdots & \vdots & \vdots & \vdots &\ddots
\end{pmatrix}
+
z^2
\begin{pmatrix}
 a_{12} & a_{13} & a_{14} & a_{15} & \dots \\
 a_{22} & 0 & 0 & 0 &\dots  \\
 a_{32} & 0 & 0 & 0 &\dots \\
 a_{42} & 0 & 0 & 0 & \dots \\
 \vdots & \vdots &\vdots &\vdots & \ddots
\end{pmatrix}
+ \cdots
\right)
 \\
&=&
\begin{pmatrix}
0 & 0 & 0 & 0 & \dots \\
a_{00} & a_{01} & a_{02} & a_{03} & \dots \\
a_{10} & a_{11} & a_{12} & a_{13} &\dots  \\
a_{20} & a_{21} & a_{22} & a_{23} &\dots \\
\vdots & \vdots & \vdots & \vdots &\ddots
\end{pmatrix}
= M_{z_1}
\begin{pmatrix}
a_{00} & a_{01} & a_{02} & \dots \\
a_{10} & a_{11} & a_{12} &\dots  \\
a_{20} & a_{21} & a_{22} & \dots \\
\vdots & \vdots & \vdots & \ddots
\end{pmatrix}.
\end{eqnarray*}
The proof of $M_{z_2}=U^*M_{G_2^*+zG_1}U$ is similar. Hence the proof.
\end{proof}
{\bf{Acknowledgement:}} The author is thankful to Professor T. Bhattacharyya for helpful suggestions and stimulating conversations.

\end{document}